\documentclass[reqno,12pt]{amsart}

\usepackage{bbm}

\numberwithin{equation}{section}

\newcommand{\1}{\mathbbm {1}}
\newcommand{\Z}{{\mathbb Z}}
\newcommand{\C}{{\mathbb C}}

\newcommand{\wh}{{\widehat{\mathfrak h}}}
\newcommand{\I}{{\mathcal I}}
\newcommand{\tI}{\widetilde{\I}}
\newcommand{\J}{{\mathcal J}}
\newcommand{\W}{{\mathcal W}}
\newcommand{\tW}{\widetilde{\W}}

\newcommand{\mraff}{\mathrm{aff}}

\newcommand{\la}{\langle}
\newcommand{\ra}{\rangle}

 \DeclareMathOperator{\ad}{ad}
 \DeclareMathOperator{\Ind}{Ind}

\newtheorem{thm}{Theorem}[section]

\newtheorem{rmk}[thm]{Remark}

\begin{document}

\title
{The structure of parafermion vertex operator algebras}

\author[C. Dong]{Chongying Dong}
\address{Department of Mathematics, University of California, Santa Cruz,
CA 95064 \& School of Mathematics, Sichuan University, China} \email{dong@math.ucsc.edu}

\author[C. H. Lam]{Ching Hung Lam}
\address{Department of Mathematics and National Center for Theoretical
Sciences, National Cheng Kung University, Tainan, Taiwan 701}
\email{chlam@mail.ncku.edu.tw}

\author[Q. Wang]{Qing Wang}
\address{School of Mathematical Sciences, Xiamen University, Xiamen 361005,
 China}
\email{qingwang@xmu.edu.cn}

\author[H. Yamada]{Hiromichi Yamada}
\address{Department of Mathematics, Hitotsubashi University, Kunitachi,
Tokyo 186-8601, Japan}
\email{yamada@math.hit-u.ac.jp}


\begin{abstract}
It is proved that the parafermion vertex operator algebra associated to the
irreducible highest weight module for the affine Kac-Moody algebra
$A_1^{(1)}$ of level $k$ coincides with a certain $W$-algebra. In particular,
a set of generators for the  parafermion vertex operator algebra is determined.
\end{abstract}

\maketitle

\section{Introduction}

The coset construction initiated in \cite{GKO} is another major way besides
the orbifold theory to construct new conformal field theories from given
ones. The coset constructions have been used to establish the unitarity of
the discrete series for the Viraoro algebras \cite{GKO} and to produce many
important conformal field theories associated to affine Kac-Moody algebras
(see for example \cite{DL, GQ, KL, X1, X2, ZF}). The coset construction in
the language of vertex operator algebra can be stated as follows (see
\cite[Section 5]{FZ}, \cite[Section 3.11]{LL}): Let $V=(V,Y,\1,\omega)$ be
a vertex operator algebra and $U=(U,Y,\1,\omega^1)$ a vertex operator
subalgebra of $V$. Then the commutant $U^c=\{v\in V|u_nv=0 \ for \ u\in U,
n\geq 0\}$ of $U$ in $V$ is another vertex operator subalgebra of $V$ with
Virasoro vector $\omega^2 = \omega-\omega^1$ under a suitable assumption.
The commutant $U^c$ is called the coset vertex operator algebra associated
to the pair $V \supset U$.

The parafermion algebras investigated in \cite{ZF} are essentially the
$Z$-algebras introduced and studied earlier in \cite{LP,LW1,LW2}, as
clarified in \cite{DL}. It was proved that the parafermion algebras are
generalized vertex operator algebras \cite{DL}.  A generalized vertex
operator algebra as a vector space is a direct sum of a vertex operator
algebra with some of its modules satisfying certain conditions. We call the
vertex operator algebras in the parafermion algebras the parafermion vertex
operator algebras, which are denoted by $K_0$ in \cite{DLY2}.

Let $k \ge 2$ be an integer and $L(k,0)$ the level $k$ irreducible highest
weight module for the affine Kac-Moody algebra $A_1^{(1)}$. It turns out
that $K_0$ is exactly the commutant of the Heisenberg vertex operator
subalgebra in the vertex operator algebra $L(k,0)$. Certain $W$-algebras
have been constructed and realized as vertex operator subalgebras of the
commutant $N_0$ of the Heisenberg vertex operator subalgebra in the vertex
operator algebra $V(k,0)$, where $V(k,0)$ denotes the level $k$ Weyl module
for the affine Kac-Moody algebra $A_1^{(1)}$ \cite{BEHHH, H}. We remark
that $L(k,0)$ is the simple quotient of $V(k,0)$ and $K_0$ is the simple
quotient of $N_0$. Motivated by a desire to understand the moonshine vertex
operator algebra \cite{FLM} better, the vertex operator algebras $K_0$ and
$N_0$ have been further investigated in \cite{DLY1, DLY2}. In particular,
several conjectures regarding the generators of $N_0$ and $K_0$, and the
rationality and the $C_2$-cofiniteness of $K_0$ were made.

In this paper we will determine a set of generators for both $N_0$ and
$K_0$, giving a positive answer to a conjecture in \cite{DLY2}. In
particular, we obtain that $N_0$ is in fact equal to the $W$-algebra
$W(2,3,4,5)$ \cite{BEHHH}. We will also prove that the unique maximal ideal
of $N_0$ is generated by a single vector and consequently determine the
structure of the simple quotient $K_0$ of $N_0$. If $k\le 6$, these results
has been obtained in \cite{DLY2}.

It has been conjectured that $K_0$ is a rational and $C_2$-cofinite vertex
operator algebra. Again this has been established in the case $k\le 6$
\cite{DLY2}. It is widely believed that if a vertex operator algebra $V$ is
rational and its subalgebra $U$ is rational, then the commutant $U^c$ of
$U$ in $V$ is also rational. Although the Heisenberg vertex operator
algebra is not rational, the commutant $K_0$ can be regraded as the
commutant of a lattice vertex operator algebra associated to a rank one
lattice inside $L(k,0)$ \cite[Proposition 4.1]{DLY2}. Also, $K_0$ occurs as
the commutant of a tensor product of vertex operator algebras associated to
Virasoro algebras of discrete series in a lattice vertex operator algebra
associated to a lattice of type $\sqrt{2}A_{k-1}$ \cite[Theorem 4.2]{LY}.
These facts should explain why the rationality of $K_0$ is expected. One
can study the commutant of the Heisenberg vertex operator algebra inside
the affine vertex operator algebra for any affine Kac-Moody algebra and
obtain a class of rational vertex operator algebras \cite{GQ}.

The paper is organized as follows. In Section 2, we recall the construction
of the vertex operator algebra $V(k,0)$ associated to the affine Kac-Moody
algebra $A_1^{(1)}$ from \cite{FZ}. Moreover, we consider its subalgebra
$V(k,0)(0)$, which is the kernel of the action of $h$ on $V(k,0)$. Here we
use the standard basis $\{h,e,f\}$ for the Lie algebra $sl_2$. We give a
set of generators for $V(k,0)(0)$. This result is the foundation for the
study of generators for both $N_0$ and $K_0.$ In Section 3, we define
vertex operator algebra $N_0$ and prove that $N_0$ coincides with its
subalgebra $W(2,3,4,5)$ generated by the Virasoro vector and Virasoro
primary vectors $W^i$ of weight $i$ for $i=3,4,5$ by showing that $N_0$ is
in fact generated by the Virasoro vector and $W^3$. This solves a
conjecture given in \cite{DLY2}. Section 4 is devoted to the unique maximal
ideal of $N_0$ and its simple quotient $K_0$. The result in this section is
that the maximal ideal is generated by $f(0)^{k+1}e(-1)^{k+1}\1$.
Furthermore, we prove an important property of the vector
$f(0)^{k+1}e(-1)^{k+1}\1$. These results settle down another conjecture in
\cite{DLY2}. The main idea in proving the results in this section is to use
the highest weight module theory for the finite dimensional simple Lie
algebra $sl_2$.

\section{Vertex operator algebras $V(k,0)$ and $V(k,0)(0)$}
\label{Sect:V(k,0)}

We are working in the setting of \cite{DLY2}. In particular, $\{ h, e, f\}$
is a standard Chevalley basis of $sl_2$ with $[h,e] = 2e$, $[h,f] = -2f$,
$[e,f] = h$ for the bracket, $\la,\ra$ is the normalized Killing form so
that $\la h,h \ra = 2$, $\la e,f \ra = 1$, $\la h,e \ra = \la h,f \ra = \la
e,e \ra = \la f,f \ra = 0$, and $\widehat{sl}_2 = sl_2 \otimes \C[t,t^{-1}]
\oplus \C C$ is the corresponding affine Lie algebra. Moreover, $k \ge 2$
is an integer and
\begin{equation*}
V(k,0) = V_{\widehat{sl}_2}(k,0) = \Ind_{sl_2 \otimes \C[t]\oplus \C
C}^{\widehat{sl}_2}\C
\end{equation*}
is an induced $\widehat{sl}_2$-module such that $sl_2 \otimes \C[t]$ acts
as $0$ and $C$ acts as $k$ on $\1=1$.

We denote by $a(n)$ the operator on $V(k,0)$ corresponding to the action of
$a \otimes t^n$. Then
\begin{equation}\label{eq:affine-commutation}
[a(m), b(n)] = [a,b](m+n) + m \la a,b \ra \delta_{m+n,0}k
\end{equation}
for $a, b \in sl_2$ and $m,n\in \Z$. Note that $a(n)\1 =0$ for $n \ge 0$.
The vectors
\begin{equation}\label{eq:V-basis}
h(-i_1) \cdots h(-i_p) e(-j_1) \cdots e(-j_q) f(-m_1) \cdots f(-m_r)\1,
\end{equation}
$i_1 \ge \cdots \ge i_p \ge 1$, $j_1 \ge \cdots \ge j_q \ge 1$, $m_1 \ge
\cdots \ge m_r \ge 1$ and $p,q,r \ge 0$ form a basis of $V(k,0)$.

Let $a(z) = \sum_{n \in \Z} a(n)z^{-n-1}$. Then $V(k,0)$ is a vertex
operator algebra generated by $a(-1)\1$ for $a\in sl_2$ such that
$Y(a(-1)\1,z) = a(z)$ with the vacuum vector $\1$ and the Virasoro vector
\begin{align*}
\omega_{\mraff} &= \frac{1}{2(k+2)} \Big( \frac{1}{2}h(-1)^2\1 +
e(-1)f(-1)\1 + f(-1)e(-1)\1 \Big)\\
&= \frac{1}{2(k+2)} \Big( -h(-2)\1 + \frac{1}{2}h(-1)^2\1 + 2e(-1)f(-1)\1
\Big)
\end{align*}
of central charge $3k/(k+2)$ \cite{FZ} (see \cite[Section 6.2]{LL} also).
The vector of the form \eqref{eq:V-basis} has weight $i_1 + \cdots + i_p +
j_1 + \cdots + j_q + m_1 + \cdots + m_r$. We also note that the vector of
the form \eqref{eq:V-basis} is an eigenvector for $h(0)$ with eigenvalue
$2(q-r)$.

For $\lambda \in 2\Z$, set
\begin{equation*}
V(k,0)(\lambda)=\{v\in V(k,0)|h(0)v = \lambda v\}.
\end{equation*}
Then we have an eigenspace decomposition for $h(0)$:
\begin{equation}\label{eq:V-dec}
V(k,0)=\oplus_{\lambda \in 2\Z}V(k,0)(\lambda).
\end{equation}
Since $[h(0), Y(u,z)]=Y(h(0)u,z)$ for $u \in V(k,0)$ by the definition of
affine vertex operator algebra, we see that $V(k,0)(0)$ is a vertex
operator subalgebra of $V(k,0)$ with the same Virasoro vector
$\omega_{\mraff}$ and each $V(k,0)(\lambda)$ is a module for $V(k,0)(0)$.

Our first theorem is on a set of generators for $V(k,0)(0)$, which will be
fundamental in the study of generators of $N_0$ and $K_0$ later.

\begin{thm}\label{generator1} The vertex operator algebra
$V(k,0)(0)$ is generated by two vectors $h(-1)\1$ and $f(-2)e(-1)\1$.
\end{thm}

\begin{proof} First of all, note that $V(k,0)(0)$ is spanned
by the vectors
\begin{equation}\label{eq:span}
h(-i_1)\cdots h(-i_p)f(-m_1)e(-n_1)\cdots f(-m_s)e(-n_s)\1
\end{equation}
for $i_j, m_j, n_j>0$ and $p,s \ge 0$. Let $U$ be the vertex operator
subalgebra generated by $h(-1)\1$ and $f(-2)e(-1)\1$. Moreover, let
$V(k,0)(0,t)$ be the subspace spanned by the vectors in \eqref{eq:span}
with $s \le t$. We prove by induction on $t$ that $V(k,0)(0,t)$ is a
subspace of $U$. We first consider the case $t=1$.

Since $(h(-1)\1)_n = h(n)$, we have $h(-i_1)\cdots h(-i_p)v \in U$ if $v
\in U$. Thus, in order to show that $V(k,0)(0,1)$ is a subspace of $U$, it
suffices to verify that $f(-m)e(-n)\1 \in U$ for $m,n > 0$. In fact, we
prove by induction on $n$ that $f(-n+i)e(-i)\1 \in U$ for $n \ge 2$ and $1
\le i \le n-1$.

We have $h(1)f(-2)e(-1)\1 = -2f(-1)e(-1)\1$. Hence $f(-1)e(-1)\1 \in U$ and
so $\omega_{\mraff} \in U$. Set $L_{\mraff}(n) = (\omega_{\mraff})_{n+1}$,
that is,
\begin{equation*}
Y(\omega_{\mraff},z)=\sum_{n\in \Z}L_{\mraff}(n)z^{-n-2}.
\end{equation*}
Then
\begin{equation*}
[L_{\mraff}(m),a(n)]=-na(m+n)
\end{equation*}
for $m,n\in\Z$ and $a\in sl_2$. Since $L_{\mraff}(-1)\1 = 0$, it follows
that
\begin{equation}\label{eq:Laff-1}
L_{\mraff}(-1)f(-m+i)e(-i)\1 = (m-i)f(-m-1+i)e(-i)\1 + if(-m+i)e(-i-1)\1
\end{equation}
for any $m,i \in \Z$. In particular,
\begin{equation}\label{eq:Laff-1bis}
L_{\mraff}(-1)f(-1)e(-1)\1 = f(-2)e(-1)\1 + f(-1)e(-2)\1.
\end{equation}
This implies that $f(-1)e(-2)\1 \in U$.

We have shown that $f(-n+i)e(-i)\1 \in U$ for $1 \le i \le n-1$ in the
cases $n = 2, 3$. Now, let $n \ge 3$ and assume that $f(-m+i)e(-i)\1 \in U$
for $2 \le m \le n$ and $1 \le i \le m-1$. We want to show that
$f(-n-1+i)e(-i)\1 \in U$ for $1 \le i \le n$.

We need the following identity
\begin{equation}\label{eq:gene}
(u_lv)_m=\sum_{j\geq 0}(-1)^j \binom{l}{j} u_{l-j}v_{m+j} -\sum_{j\geq
0}(-1)^{l+j} \binom{l}{j} v_{m+l-j}u_{j}
\end{equation}
for $u,v\in V(k,0)$ and $l,m\in \Z$, which is a consequence of the Jacobi
identity of vertex operator algebra. Applying \eqref{eq:gene} to
$(f(-2)e(-1)\1)_1 = ((f(-1)\1)_{-2}e(-1)\1)_1$, we have
\begin{equation*}
(f(-2)e(-1)\1)_1=\sum_{j \ge 0}(j+1)f(-2-j)e(1+j)-\sum_{j \ge
0}(j+1)e(-1-j)f(j).
\end{equation*}
Notice that $e(1+j)f(-n+1)e(-1)\1 = 0$ if $j \ge n-1$ and
$f(j)f(-n+1)e(-1)\1 = 0$ if $j \ge 2$. Thus we have
\begin{equation}\label{eq:fe1}
\begin{split}
& (f(-2)e(-1)\1)_{1}f(-n+1)e(-1)\1\\
& \qquad = \sum_{1 \le j \le n-1} jf(-1-j)e(j)f(-n+1)e(-1)\1\\
& \qquad \quad  -\sum_{j=0,1} (j+1)e(-1-j)f(j)f(-n+1)e(-1)\1.
\end{split}
\end{equation}

We consider the first summation of the right hand side of \eqref{eq:fe1}.
Since $n \ge 3$, we have
\begin{equation*}
\begin{split}
& f(-1-j)e(j)f(-n+1)e(-1)\1\\
& \qquad = f(-1-j)h(-n+1+j)e(-1)\1\\
& \qquad = 2f(-n)e(-1)\1+h(-n+1+j)f(-1-j)e(-1)\1
\end{split}
\end{equation*}
for $1 \leq j \le n-2$ and
\begin{equation*}
\begin{split}
& f(-n)e(n-1)f(-n+1)e(-1)\1\\
& \qquad = f(-n)\big(h(0)+(n-1)k + f(-n+1)e(n-1)\big)e(-1)\1\\
& \qquad = f(-n)h(0)e(-1)\1 + (n-1)kf(-n)e(-1)\1\\
& \qquad = (2+(n-1)k)f(-n)e(-1)\1.
\end{split}
\end{equation*}

As to the second summation of the right hand side of \eqref{eq:fe1}, we
have
\begin{equation*}
\begin{split}
& e(-1)f(0)f(-n+1)e(-1)\1\\
& \qquad = -e(-1)f(-n+1)h(-1)\1\\
& \qquad = -h(-n)h(-1)\1 - f(-n+1)e(-1)h(-1)\1\\
& \qquad = -h(-n)h(-1)\1 + 2f(-n+1)e(-2)\1 - f(-n+1)h(-1)e(-1)\1\\
& \qquad = -h(-n)h(-1)\1 + 2f(-n+1)e(-2)\1\\
& \qquad \quad\, - 2f(-n)e(-1)\1 - h(-1)f(-n+1)e(-1)\1,
\end{split}
\end{equation*}

\begin{equation*}
\begin{split}
& e(-2)f(1)f(-n+1)e(-1)\1\\
& \qquad = ke(-2)f(-n+1)\1\\
& \qquad = kh(-n-1)\1 + kf(-n+1)e(-2)\1.
\end{split}
\end{equation*}

Thus the identity \eqref{eq:fe1} becomes
\begin{equation*}
\begin{split}
& (f(-2)e(-1)\1)_{1}f(-n+1)e(-1)\1\\
& \qquad = \big( (n-1)(n+(n-1)k)+2 \big)f(-n)e(-1)\1 - 2(k+1)f(-n+1)e(-2)\1
+ u
\end{split}
\end{equation*}
for some $u \in U$ by the induction assumption. Furthermore,
\begin{equation*}
f(-n+1)e(-2)\1 = L_{\mraff}(-1)f(-n+1)e(-1)\1 - (n-1)f(-n)e(-1)\1
\end{equation*}
by \eqref{eq:Laff-1}. Therefore,
\begin{equation*}
\begin{split}
& (f(-2)e(-1)\1)_{1}f(-n+1)e(-1)\1\\
& \qquad = \big( (n-1)(nk + n + k + 2) + 2\big)f(-n)e(-1)\1\\
& \qquad \quad\, - 2(k+1)L_{\mraff}(-1)f(-n+1)e(-1)\1 + u.
\end{split}
\end{equation*}
Since $f(-n+1)e(-1)\1 \in U$ by the induction assumption, this implies that
$f(-n)e(-1)\1$ lies in $U$. As a result $f(-n-1+i)e(-i)\1 \in U$ for $1 \le
i \le n$ by \eqref{eq:Laff-1} and the induction on $n$ is complete. Thus
$V(k,0)(0,1)$ is a subspace of $U$.

We now assume that $V(k,0)(0,t)$ is a subspace of $U$ and show that
$V(k,0)(0,t+1)$ is also a subspace of $U$. We use the expression
\begin{equation}\label{eq:uv}
(u_{-m-1}v_{-n-1}\1)_{-1} = u_{-m-1}v_{-n-1} + \sum_{i \ge
0}c_iu_{-m-n-2-i}v_i + \sum_{i \ge 0}d_iv_{-m-n-2-i}u_i
\end{equation}
for $u, v \in V(k,0)$ and $m,n \ge 0$, where $c_i,d_i$ are some constants.
Indeed,
\begin{equation}\label{eq:uv-aux}
\begin{split}
(u_{-m-1}v_{-n-1}\1)_{-1}
&= \sum_{i \ge 0} (-1)^i \binom{-m-1}{i}u_{-m-1-i}(v_{-n-1}\1)_{-1+i}\\
& \quad\, - \sum_{i \ge 0} (-1)^{-m-1+i}\binom{-m-1}{i}
(v_{-n-1}\1)_{-m-2-i}u_i
\end{split}
\end{equation}
by \eqref{eq:gene}. Similarly,
\begin{equation*}
\begin{split}
(v_{-n-1}\1)_{-1+i} &= \sum_{j \ge 0} (-1)^j \binom{-n-1}{j}
v_{-n-1-j}\1_{-1+i+j}\\
& \quad\, - \sum_{j \ge 0} (-1)^{-n-1+j} \binom{-n-1}{j} \1_{-n-2+i-j}v_j.
\end{split}
\end{equation*}
Since $i,j \ge 0$ and $\1_r = \delta_{r,-1}$, we have that $\1_{-1+i+j} =
0$ unless $i=j=0$. Then the first summation of the right hand side of
\eqref{eq:uv-aux} is $u_{-m-1}v_{-n-1} + \sum_{i \ge 0}c_iu_{-m-n-2-i}v_i$
for some constants $c_i$. By a similar argument, we see that the second
summation of the right hand side of \eqref{eq:uv-aux} is $\sum_{i \ge
0}d_iv_{-m-n-2-i}u_i$ for some constants $d_i$. Thus \eqref{eq:uv} holds.

Now, take $u=f(-1)\1$ and $v=e(-1)\1$ to obtain
\begin{equation}\label{eq:fe-1}
\begin{split}
& (f(-m-1)e(-n-1)\1)_{-1}\\
& \qquad = f(-m-1)e(-n-1)\\
& \qquad \quad\, + \sum_{i\ge 0}c_if(-m-n-2-i)e(i) + \sum_{i \ge
0}d_ie(-m-n-2-i)f(i).
\end{split}
\end{equation}

Let $w=f(-m_1)e(-n_1)\cdots f(-m_t)e(-n_t)\1\in V(k,0)(0,t)$. Using the
commutation relation \eqref{eq:affine-commutation} and the property that
$e(i)\1 = f(i)\1 = 0$ for $i \ge 0$, we can show that if we express each
$f(-m-n-2-i)e(i)w$ and $e(-m-n-2-i)f(i)w$ for $i \ge 0$ as linear
combinations of vectors in \eqref{eq:span}, then these vectors are
contained in $V(k,0)(0,t)$. Recall that $V(k,0)(0,t) \subset U$ by the
induction assumption. Then $(f(-m-1)e(-n-1)\1)_{-1}w \in U$ and it follows
from \eqref{eq:fe-1} that $f(-m-1)e(-n-1)w \in U$. Thus $V(k,0)(0,t+1)
\subset U$, as desired.
\end{proof}

\begin{rmk}\label{Rem:another-generattor}
We can replace $f(-2)e(-1)\1$ by $f(-1)e(-2)\1 - f(-2)e(-1)\1$ in Theorem
\ref{generator1}. Indeed,
\begin{equation*}
h(1)(f(-1)e(-2)\1 - f(-2)e(-1)\1) = 2h(-2)\1 + 4f(-1)e(-1)\1.
\end{equation*}
Hence by \eqref{eq:Laff-1bis}, $h(-1)\1$ and $f(-1)e(-2)\1 - f(-2)e(-1)\1$
generate the vertex operator algebra $V(k,0)(0)$. In Section
\ref{Sect:maximal-ideal-tI}, we will consider an automorphism $\theta$ of
the vertex operator algebra $V(k,0)$ of order $2$, which leaves $V(k,0)(0)$
invariant. The above set of generators for $V(k,0)(0)$ is suitable to the
action of the automorphism $\theta$, since $\theta$ acts as $-1$ on those
generators. The vector $f(-1)e(-2)\1 - f(-2)e(-1)\1$ is also closely
related to $W^3$ (see the proof of Theorem \ref{generator2} below).
\end{rmk}

\section{Vertex operator algebra $N_0$ and $W$-algebra}
\label{walgebra}

There are two subalgebras $\wh = \C h \otimes \C[t,t^{-1}] \oplus \C C$ and
$\wh_\ast = ( \oplus_{n \ne 0} \C h \otimes t^n) \oplus \C C$ of
$\widehat{sl}_2$.  The subspace $V_{\wh}(k,0)$ spanned by $h(-i_1) \cdots
h(-i_p)\1$ for $i_1 \ge \cdots \ge i_p \ge 1$ and $p \ge 0$ is a vertex
operator subalgebra of $V(k,0)$ associated to the Heisenberg algebra
$\wh_\ast$ of level $k$ with the Virasoro vector
\begin{equation}\label{eq:omega_gamma}
\omega_{\gamma} = \frac{1}{4k} h(-1)^2\1
\end{equation}
of central charge $1$.

Now, $V(k,0)$ and each $V(k,0)(\lambda)$, $\lambda \in 2\Z$ are completely
reducible as a $V_{\wh}(k,0)$-module. More precisely,
\begin{equation}\label{eq:dec-Heisenberg}
V(k,0) = \oplus_{\lambda\in 2\Z} M_{\wh}(k,\lambda) \otimes N_\lambda,
\end{equation}
\begin{equation}\label{eq:dec-Heisenberg1}
V(k,0)(\lambda)= M_{\wh}(k,\lambda) \otimes N_\lambda,
\end{equation}
where $M_{\wh}(k,\lambda)$ denotes an irreducible highest weight module
for $\wh$ with a highest weight vector $v_\lambda$ such that $h(0)v_\lambda
= \lambda v_\lambda$ and
\begin{equation*}
N_\lambda = \{ v \in V(k,0)\,|\, h(m)v = \lambda\delta_{m,0}v \text{ for
} m \ge 0\}.
\end{equation*}

Note that $M_{\wh}(k,0)$ can be identified  with $V_{\wh}(k,0)$ and $N_0$
is the commutant \cite[Theorem 5.1]{FZ} of $V_{\wh}(k,0)$ in $V(k,0)$. The
commutant $N_0$ is a vertex operator algebra with the Virasoro vector
$\omega = \omega_{\mraff} - \omega_{\gamma}$;
\begin{equation}\label{eq:omega}
\omega = \frac{1}{2k(k+2)} \Big( -kh(-2)\1 -h(-1)^2\1 + 2k e(-1)f(-1)\1
\Big),
\end{equation}
whose central charge is $3k/(k+2) - 1 = 2(k-1)/(k+2)$. Since the Virasoro
vector of $V_{\wh}(k,0)$ is $\omega_{\gamma}$, we have $N_0 = \{ v \in
V(k,0)\,|\, (\omega_\gamma)_0 v = 0\}$ \cite[Theorem 5.2]{FZ}. It is clear
that the weight of $v$ in $N_0$ agrees with that in $V(k,0)$, since
$\omega_1 v = (\omega_{\mraff})_1 v$ for $v \in N_0$.

The dimension of weight $i$ subspace $(N_0)_{(i)}$ is $2$, $4$ and $6$ for
$i=3$, $4$ and $5$, respectively. It is known that there is up to a scalar
multiple, a unique Virasoro primary vector $W^i$ in $(N_0)_{(i)}$ for $i =
3,4,5$ \cite[Section 2]{DLY2}. Here a Virasoro primary vector of weight $i$
means that $\omega_2 W^i = \omega_3 W^i = 0$ and $\omega_1 W^i = i W^i$. As
in \cite{DLY2}, we take
\begin{equation*}\label{eq:W3}
\begin{split}
W^3 &= k^2 h(-3)\1 + 3 k h(-2)h(-1)\1 +
2h(-1)^3\1 - 6k h(-1)e(-1)f(-1)\1 \\
& \quad + 3 k^2e(-2)f(-1)\1 - 3 k^2e(-1)f(-2)\1,
\end{split}
\end{equation*}

\begin{equation*}
\begin{split}
W^4 &= -2 k^2 (k^2+k+1) h(-4)\1 -8 k (k^2+k+1)h(-3)h(-1)\1 -k (5
k^2-6)h(-2)^2\1 \\
& \quad -\!2 k (11 k+6)h(-2)h(-1)^2\1 -\!(11 k+6)h(-1)^4\1 +4
k^2 (6 k-5)h(-2)e(-1)f(-1)\1 \\
& \quad +4 k (11 k+6)h(-1)^2e(-1)f(-1)\1 -4 k^2 (5
k+11)h(-1)e(-2)f(-1)\1 \\
& \quad +4 k^2 (5 k+11)h(-1)e(-1)f(-2)\1 +8 k^2 (k-3)
(k-2)e(-3)f(-1)\1 \\
& \quad -4 k^2 (3 k^2-3 k+8)e(-2)f(-2)\1 -2 k^2 (6
k-5)e(-1)^2f(-1)^2\1 \\
& \quad +8 k^2 (k^2+k+1)e(-1)f(-3)\1,
\end{split}
\end{equation*}

\begin{equation*}
\begin{split}
W^5 &= -2 k^3 (k^2+3 k+5)h(-5)\1 -10 k^2 (k^2+3 k+5)h(-4)h(-1)\1\\
& \quad -5 k^2 (3 k^2-4)h(-3)h(-2)\1 -5 k (7 k^2+12 k+16)h(-3)h(-1)^2\1\\
& \quad -15 k (3 k^2-4)h(-2)^2h(-1)\1 -5 k (19 k+12)h(-2)h(-1)^3\1 -2 (19
k+12)h(-1)^5\1 \\
& \quad +10 k^2 (4 k^2-7 k+8)h(-3)e(-1)f(-1)\1 +20 k^2 (10
k-7)h(-2)h(-1)e(-1)f(-1)\1 \\
& \quad +10 k (19 k+12)h(-1)^3e(-1)f(-1)\1 -5 k^2 (11
k^2-14 k+12)h(-2)e(-2)f(-1)\1 \\
& \quad -5 k^2 (17 k+64)h(-1)^2e(-2)f(-1)\1 +15 k^2 (3
k^2-4)h(-2)e(-1)f(-2)\1 \\
& \quad +5 k^2 (17 k+64)h(-1)^2e(-1)f(-2)\1 +30 k^2 (k-4)
(k-3)h(-1)e(-3)f(-1)\1 \\
& \quad -40 k^2 (k^2+3 k+5)h(-1)e(-2)f(-2)\1 -10 k^2 (10
k-7)h(-1)e(-1)^2f(-1)^2\1 \\
& \quad +10 k^2 (3 k^2+19 k+8)h(-1)e(-1)f(-3)\1 -10 k^3
(k-4) (k-3)e(-4)f(-1)\1 \\
& \quad +20 k^3 (k-4) (k-3)e(-3)f(-2)\1 +5 k^3 (10
k-7)e(-2)e(-1)f(-1)^2\1 \\
& \quad -10 k^3 (2 k^2-4 k+17)e(-2)f(-3)\1 -5 k^3 (10
k-7)e(-1)^2f(-2)f(-1)\1 \\
& \quad +10 k^3 (k^2+3 k+5)e(-1)f(-4)\1.
\end{split}
\end{equation*}

Denote by $\tW$ the subalgebra of $N_0$ generated by $\omega$, $W^3$, $W^4$
and $W^5$. Then $\tW$ coincides with $W(2,3,4,5)$ of \cite{BEHHH}. It is in
fact generated by $\omega$ and $W^3$ \cite{DLY2}.

The following theorem is essentially a conjecture in \cite{DLY2}.

\begin{thm}\label{generator2} The vertex operator
algebra $N_0$ is generated by $\omega$ and $W^3$. In particular, $N_0$
coincides with $\tW$ or $W(2,3,4,5)$.
\end{thm}

\begin{proof}
We first show that $V(k,0)(0)= M_{\wh}(k,0) \otimes N_0$ is generated by
$h(-1)\1$, $\omega$ and $W^3$. Let $U$ be the vertex operator subalgebra
generated by $h(-1)\1$, $\omega$ and $W^3$. Then $U$ contains
$f(-1)e(-1)\1$ and $\omega_{\mraff}$. Moreover, we see from the expression
of $W^3$ that $U$ contains $f(-1)e(-2)\1 - f(-2)e(-1)\1$. Hence
$f(-2)e(-1)\1 \in U$ by \eqref{eq:Laff-1bis}, and so $U$ is equal to
$V(k,0)(0)$ by Theorem \ref{generator1}.

Now, $Y(u,z_1)Y(v,z_2)=Y(v,z_2)Y(u,z_1)$ for $u \in M_{\wh}(k,0)$ and $v
\in N_0$. Since $h(-1)\1 \in M_{\wh}(k,0)$ and $\omega, W^3 \in N_0$, we
conclude that $N_0$ is generated by $\omega$ and $W^3$.
\end{proof}

\begin{rmk}\label{Rem:generation}
Since $W^3_3W^3 = 36k^3(k-2)(k+2)(3k+4)\omega$, the vector $W^3$ in fact
generates the vertex operator algebra $N_0$ by Theorem \ref{generator2} if
$k \ge 3$. In the case $k = 2$, $W^3$ is contained in a unique maximal
ideal of $N_0$ \cite[Remark 2.3]{DLY2}.
\end{rmk}

It is shown in \cite[Lemma 2.6]{DLY2} that the Zhu algebra $A(\tW)$
\cite{Zhu} is a commutative associative algebra. Thus the Zhu algebra
$A(N_0)$ is commutative and every simple $A(N_0)$-module is
one-dimensional.

\section{The maximal ideal $\tI$ of $N_0$ and $K_0$
}\label{Sect:maximal-ideal-tI} The vertex operator algebra $V(k,0)$ has a
unique maximal ideal $\J$, which is generated by a weight $k+1$ vector
$e(-1)^{k+1}\1$ \cite{K}. The quotient algebra $L(k,0) = V(k,0)/\J$ is the
simple vertex operator algebra associated to an affine Lie algebra
$\widehat{sl}_2$ of type $A_1^{(1)}$ with level $k$. The Heisenberg vertex
operator algebra $V_{\wh}(k,0)$ is again a simple subalgebra of $L(k,0)$
and $L(k,0)$ is a completely reducible $V_{\wh}(k,0)$-module. We have a
decomposition
\begin{equation}
L(k,0) = \oplus_{\lambda \in 2\Z} M_{\wh}(k,\lambda) \otimes K_\lambda
\end{equation}
as modules for $V_{\wh}(k,0)$, where
\begin{equation*}
K_\lambda = \{v \in L(k,0)\,|\, h(m)v = \lambda\delta_{m,0}v \text{ for } m
\ge 0\}.
\end{equation*}
Note that $M_{\wh}(k,0) = V_{\wh}(k,0)$ and $K_0$ is the commutant of
$V_{\wh}(k,0)$ in $L(k,0)$.

Similarly, $\J$ is completely reducible as a $V_{\wh}(k,0)$-module. Hence by
\eqref{eq:dec-Heisenberg},
\begin{equation*}
\J = \oplus_\lambda M_{\wh}(k,\lambda) \otimes (\J \cap N_\lambda).
\end{equation*}
In particular, $\tI = \J \cap N_0$ is an ideal of $N_0$ and $K_0 \cong
N_0/\tI$. It is proved in \cite[Lemma 3.1]{DLY2} that $\tI$ is the unique
maximal ideal of $N_0.$ Thus $K_0$ is a simple vertex operator algebra.

For short we still use $\omega_{\mraff}$, $\omega_\gamma$, $\omega$, $W^3$,
$W^4$ and $W^5$ of $V(k,0)$ to denote their images in $L(k,0) = V(k,0)/\J$.
Let $\W$ be the subalgebra of $K_0$ generated by $\omega$, $W^3$, $W^4$ and
$W^5$. Thus $\W$ is a homomorphic image of $\tW$. The following result is a
direct consequence of Theorem \ref{generator2}

\begin{thm}\label{generator3} The simple vertex operator algebra $K_0$
is generated by $\omega$ and $W^3$. In particular, $K_0$ coincides with
$\W$.
\end{thm}

We remark that this theorem in the case $k\le 6$ has been obtained in
\cite{DLY2} and was conjectured for general $k$. As mentioned in Remark
\ref{Rem:generation}, $K_0$ is generated by $W^3$ if $k \ge 3$, while $W^3
= 0$ in $K_0$ if $k =2$.

Next, we study the ideal $\tI$ of $N_0$ in detail. For this purpose we
recall that the Lie algebra $sl_2$ has an involution $\theta$ given by $h
\mapsto -h$, $e \mapsto f$, $f \mapsto e$. The involution $\theta$ lifts to
an automorphism of the vertex operator algebra $V(k,0)$ of order $2$
naturally. We still denote it by $\theta.$ Then $\theta \omega = \omega$,
$\theta W^3 = -W^3$,  $\theta W^4 = W^4$ and $\theta W^5 = -W^5$.

It is proved in \cite[Theorem 3.2]{DLY2} that $f(0)^{k+1}e(-1)^{k+1}\1 \in
\tI.$ The following theorem is another conjecture in \cite{DLY2}.

\begin{thm}\label{Conj:ideal-generator}
$(1)$ The unique maximal ideal $\tI$ of $N_0$ is generated by a weight
$k+1$ vector $f(0)^{k+1}e(-1)^{k+1}\1$.

$(2)$ The automorphism $\theta$ acts as $(-1)^{k+1}$ on
$f(0)^{k+1}e(-1)^{k+1}\1$.
\end{thm}

\begin{proof}
We use the finite dimensional representation theory for the simple Lie
algebra $sl_2$. In fact, $V(k,0)$ is an $sl_2$-module where $a \in sl_2$
acts as $a(0)$. Each weight subspace of the vertex operator algebra
$V(k,0)$ is a finite dimensional $sl_2$-module and $V(k,0)$ is completely
reducible as a module for $sl_2$.

We first show the assertion (1). Consider the $sl_2$-submodule $X$ of
$V(k,0)$ generated by $e(-1)^{k+1}\1$. We have $e(0)e(-1)^{k+1}\1 =0$ and
$h(0)e(-1)^{k+1}\1 = 2(k+1)e(-1)^{k+1}\1$, that is, $e(-1)^{k+1}\1$ is a
highest weight vector with highest weight $2(k+1)$ for $sl_2$. Then $X$ is
an irreducible $sl_2$-module with basis $f(0)^{i}e(-1)^{k+1}\1$, $0 \le i
\le 2(k+1)$ from the representation theory of $sl_2$. This implies that the
ideal $\J$ of the vertex operator algebra $V(k,0)$ can be generated by any
nonzero vector in $X$. In particular, $\J$ is generated by
$f(0)^{k+1}e(-1)^{k+1}\1$. Then $\J$ is spanned by $u_n
f(0)^{k+1}e(-1)^{k+1}\1$ for $u \in V(k,0)$ and $n \in \Z$ by
\cite[Corollary 4.2]{DM} or \cite[Proposition 4.1]{L}. Since $[h(0),
Y(u,z)] = Y(h(0)u,z)$ and $h(0)f(0)^{k+1}e(-1)^{k+1}\1 = 0$, it follows
that
\begin{equation*}
h(0)u_n f(0)^{k+1}e(-1)^{k+1}\1 = (h(0)u)_n f(0)^{k+1}e(-1)^{k+1}\1.
\end{equation*}
Thus we see from \eqref{eq:V-dec} that $u_n f(0)^{k+1}e(-1)^{k+1}\1 \in \J
\cap V(k,0)(0)$ if and only if $u \in V(k,0)(0)$. Let $u = v\otimes w\in
V(k,0)(0) = M_{\wh}(k,0) \otimes N_0$ with $v \in M_{\wh}(k,0)$ and $w\in
N_0$. Then $Y(u,z)=Y(v,z)\otimes Y(w,z)$ acts on $M_{\wh}(k,0) \otimes
N_0$. As a result we have that $\tI$ is spanned by $w_n
f(0)^{k+1}e(-1)^{k+1}\1$ for $w \in N_0$ and $n \in \Z$. That is, the ideal
$\tI$ of the vertex operator algebra $N_0$ is generated by
$f(0)^{k+1}e(-1)^{k+1}\1$. Thus (1) holds.

As to the assertion (2), we prove a more general result here:
\begin{equation}\label{eq:theta-action}
\theta(f(0)^ie(-1)^i\1)=(-1)^if(0)^ie(-1)^i\1
\end{equation}
for any positive integer $i$. Let $U$ be the irreducible $sl_2$-submodule
of $V(k,0)$ generated by the highest weight vector $e(-1)^i\1$ with highest
weight $2i$ for $sl_2$. Then $U$ has a basis $f(0)^je(-1)^i\1$, $0\le j\le
2i$. We express $f(0)^{2i}e(-1)^i\1$ as a linear combination of the vectors
of the form \eqref{eq:V-basis}. Let $v$ be a vector of the form
\eqref{eq:V-basis}. Then $h(0)v = 2(q-r)v$ and
\begin{equation*}
L_{\mraff}(0)v = (i_1 + \cdots + i_p + j_1 + \cdots + j_q + m_1 + \cdots +
m_r)v.
\end{equation*}
Since $h(0) f(0)^{2i}e(-1)^i\1 = -2if(0)^{2i}e(-1)^i\1$ and
$L_{\mraff}(0)f(0)^{2i}e(-1)^i\1 = if(0)^{2i}e(-1)^i\1$, we obtain that
\begin{equation}\label{eq:fecf}
f(0)^{2i}e(-1)^i\1=c_if(-1)^i\1
\end{equation}
for some constant $c_i$. In fact, this can be also seen in a different way.
We consider the inner derivation $(\ad f(0))x = [f(0),x]$. Note that
$f(0)\1 = 0$, $(\ad f(0)) e(-1) = -h(-1)$, $(\ad f(0))^2 e(-1) = -2f(-1)$
and $(\ad f(0))^s e(-1) = 0$ for $s \ge 3$. Hence
\begin{align*}
f(0)^{2i}e(-1)^i\1 &= \big((\ad f(0))^{2i}e(-1)^i\big)\1\\
&= a_i \big((\ad f(0))^2 e(-1)\big)^i\1
\end{align*}
for some positive integer $a_i$. That is, $c_i = (-1)^i 2^i a_i$. More
precisely, we have
\begin{equation*}
a_i = \prod_{m=0}^{i-1} \binom{2i-2m}{2} = \frac{(2i)!}{2^i}.
\end{equation*}

Let $j$ be an integer such that $0\le j\le 2i$. Then
\begin{equation*}
e(0)^jf(0)^{2i}e(-1)^i\1 = c_ie(0)^jf(-1)^i\1
\end{equation*}
by \eqref{eq:fecf}. Moreover, one can obtain from the highest weight module
structure for $sl_2$ that
\begin{equation*}
e(0)^jf(0)^{2i}e(-1)^i\1=\frac{(2i)!j!}{(2i-j)!}f(0)^{2i-j}e(-1)^i\1.
\end{equation*}

In the case $j = i$, the above two equations imply that
\begin{equation*}
\begin{split}
f(0)^{i}e(-1)^i\1 &= \frac{c_i}{(2i)!} e(0)^if(-1)^i\1\\
&= (-1)^i e(0)^if(-1)^i\1.
\end{split}
\end{equation*}
Since $\theta(f(0)^ie(-1)^i\1) = e(0)^if(-1)^i\1$, \eqref{eq:theta-action}
holds.
\end{proof}

\begin{rmk}\label{Rem:final}
The vector $f(0)^{k+1}e(-1)^{k+1}\1$ is a scalar multiple of $W^3$, $W^4$
or $W^5$ in the case $k = 2$, $3$ or $4$ \cite[Section 5]{DLY2}.
\end{rmk}

\section*{Acknowledgments}
Chongying Dong was partially supported by NSF grants
and a research grant from University of California at Santa Cruz, Ching
Hung Lam was partially supported by  NSC grant 95-2115-M-006-013-MY2 of Taiwan, Hiromichi Yamada was partially
supported by JSPS Grant-in-Aid for Scientific Research No. 17540016.

\end{document}